\documentclass[11pt]{amsart}
\usepackage{tcolorbox}
\usepackage[utf8]{inputenc}
\usepackage[T1]{fontenc}
\usepackage{amsmath,amsthm,amssymb}
\usepackage{thmtools}
\usepackage{hyperref}
\hypersetup{colorlinks=false, linkcolor=blue}
\usepackage{cleveref}
\usepackage{stmaryrd}
\usepackage{dsfont}
\usepackage{calligra,mathrsfs}
\usepackage[shortlabels]{enumitem}
\usepackage[backend=biber,safeinputenc]{biblatex}
\AtEveryBibitem{\clearfield{url}}

\newcommand{\af}{\mathfrak a}
\newcommand{\bfr}{\mathfrak b}
\newcommand{\cf}{\mathfrak c}
\newcommand{\pf}{\mathfrak p}
\newcommand{\qf}{\mathfrak q}
\newcommand{\mf}{\mathfrak m}
\newcommand{\nf}{\mathfrak n}
\newcommand{\Jf}{\mathfrak{J}}
\newcommand{\tb}{\mathbf{t}}
\newcommand{\ub}{\mathbf{u}}
\newcommand{\ab}{\mathbf{a}}
\newcommand{\bb}{\mathbf{b}}
\newcommand{\cb}{\mathbf{c}}
\newcommand{\hb}{\mathbf{h}}
\newcommand{\eb}{\mathbf{e}}

\newcommand{\C}{\mathbb C}
\newcommand{\F}{\mathbb F}
\newcommand{\N}{\mathbb N}
\newcommand{\Q}{\mathbb Q}
\newcommand{\R}{\mathbb R}
\newcommand{\Z}{\mathbb Z}

\renewcommand{\P}{\mathbb P}

\DeclareMathOperator{\fpt}{fpt}
\DeclareMathOperator{\ft}{ft}
\DeclareMathOperator{\lct}{lct}
\DeclareMathOperator{\Spec}{Spec}
\DeclareMathOperator{\Hom}{Hom}
\DeclareMathOperator{\ini}{in}
\DeclareMathOperator{\Tor}{Tor}
\DeclareMathOperator{\Ass}{Ass}
\DeclareMathOperator{\Sing}{Sing}

\DeclareMathOperator{\codim}{ht}
\DeclareMathOperator{\ord}{ord}

\newcommand{\ceil}[1]{\lceil #1\rceil}
\newcommand{\norm}[1]{\left\lVert#1\right\rVert}

\renewcommand{\char}{\textup{char }}

\renewcommand{\bar}{\overline}

\theoremstyle{plain}
\newtheorem{theorem}{Theorem}[section]
\newtheorem{prop}[theorem]{Proposition}
\newtheorem{lemma}[theorem]{Lemma}
\newtheorem{cor}[theorem]{Corollary}

\newtheorem{question}[theorem]{Question}
\newtheorem{notation}[theorem]{Notation}

\newtheorem*{theorem*}{Theorem}
\newtheorem*{mainthm}{Main Theorem}
\newtheorem*{prop*}{Proposition}

\theoremstyle{definition}
\newtheorem{defn}[theorem]{Definition}
\newtheorem{example}[theorem]{Example}
\newtheorem{remark}[theorem]{Remark}
\newtheorem{defn-prop}[theorem]{Definition-Proposition}

\crefname{prop}{proposition}{propositions}
\crefname{defn-prop}{definition-proposition}{definition-propositions}
\Crefname{defn-prop}{Definition-Proposition}{Definition-Propositions}
\Crefname{prop}{Proposition}{Propositions}
\crefname{defn}{definition}{definitions}
\Crefname{defn}{Definition}{Definitions}
\crefname{conj}{conjecture}{conjectures}
\Crefname{conj}{Conjecture}{Conjectures}

\begin{document}

\title{Extremal F-Thresholds in Regular Local Rings}

\begin{abstract}
    Let $(R, \mf)$ be a regular local ring of characteristic $p > 0$.
    Among all proper ideals $\af\subseteq R$ with a fixed order of vanishing $\ord_{\mf}(\af)$, we classify the ideals for which the $F$-threshold $\ft^{\mf}(\af)$ is minimal. 
\end{abstract} 
\author{Benjamin Baily}
\thanks{The author was supported by NSF grant DMS-2101075, NSF RTG grant DMS-1840234, and the Simons Dissertation Fellowship.}
\maketitle
\section{Introduction}

Let $(R, \mf)$ be a regular local ring of characteristic $p>0$ and let $\af$ denote a proper ideal of $R$. Assume that $R$ is $F$-finite -- that is, we assume that the Frobenius map $F:R\to R$ is finite -- we will relax this assumption later. The $F$-pure threshold of $(R, \af)$ is a nonnegative rational number \cite{blickle_discreteness_2008} measuring the singularities of the pair $(R, \af)$ at the point $\mf$ in $\Spec R$. In this setting, the $F$-pure threshold can be computed as
\[
\fpt(\af) = \sup \left\{\frac{t}{p^e}: \af^t\not\subseteq \mf^{[p^e]}\right\}, \text{ where }I^{[p^e]} := \sum_{z\in I} z^{p^e}R.
\]
The idea is that smaller values of the $F$-pure threshold correspond to ``worse singularities'' of the closed subscheme of $\Spec R$ defined by $\af$. For example, if $\af = fR$ is principal, then $\fpt(f) \leq 1$ with equality if and only $R/fR$ is Frobenius split; in this case, $R/fR$ must be reduced.

If $d = \ord_{\mf}(\af)$ denotes the greatest positive integer such that $\af\subseteq \mf^{d}$ and $n$ is the dimension of $R$, then $\fpt(\af)$ is bounded below by $\frac{1}{d}$ and bounded above by $\min(\frac{n}{d}, 1)$; see \cite{takagi_f-pure_2004}. In this paper, we describe when the $F$-pure threshold achieves the lower end of this range. 
\begin{mainthm}[\Cref{thm:E1-char-p}]
    Let $(R, \mf)$ be an $F$-finite regular local ring of characteristic $p > 0$ and $\af\subseteq R$ a proper ideal. Let $d = \ord_\mf(\af)$ and write $d = qs$ where $q$ is a power of $p$ and $s$ is coprime to $p$. Then $\fpt(\af) = \frac{1}{d}$ if and only if $\af$ is a principal ideal and there exists $g$ in $\mf^{[q]}$ such that $\af = g^sR$.
\end{mainthm}

For a more general statement and a proof, see Section \ref{sec:proof}. In particular, we show that the $F$-finite hypothesis can be weakened to the assumption that the formal fiber $\widehat{R}\otimes_R \dfrac{R_{(\pi)}}{\pi R_{(\pi)}}$ is reduced for all prime elements $\pi$ in $R$, which is satisfied by any excellent local ring. The hypothesis on the height-1 formal fibers of $R$ cannot be weakened further; see \Cref{example:formal-fiber-ctrexmpl}.

Other authors have considered a related problem for squarefree homogeneous polynomials over an algebraically closed field $k$ of characteristic $p > 0$; see \cite{kadyrsizova_lower_2022}. For $f$ in $k[x_0,\dots, x_n]$ squarefree and homogeneous of degree $d$, one has $\fpt(f) \geq \frac{1}{d-1}$ with equality if and only if $d = p^e + 1$ and $f\in \mf^{[p^e]}$. In the absence of homogeneity, however, we cannot improve on the lower bound $\fpt(f)\geq \frac{1}{d}$ even if we assume that $f$ is squarefree: the binomial $x_1^d + x_2^N$ is reduced whenever $d,N$ are coprime and $\fpt(x_1^d + x_2^N)$ converges to $\frac{1}{d}$ as $N\to \infty$ (\Cref{example:E1-char-p-reduced}).

The characteristic zero analog of the main theorem concerning the log canonical threshold (lct) is well-understood. If $(R, \mf)$ is an excellent regular local ring of equal characteristic zero and $\af\subseteq R$ an ideal, it is known to experts that $\lct(\af)\geq \frac{1}{\ord_{\mf}(\af)}$ with equality if and only if $\af =x^dR$ for some $d$ in $\Z^+$ and some $x$ in $\mf\setminus \mf^2$. More generally, for the germ at $0$ of a plurisubharmonic function $u:\C^n\to \C$ with Lelong number $\nu(u)=1$, one has $\lct(u) \geq 1$ with equality if and only if $u = \log |z_1| + v$ where $z_1$ is a local coordinate and $\nu(v) = 0$; see \cite{guan_lelong_2015}.

A straightforward consequence of the main theorem is that the lower bound $\frac{1}{\ord_{\mf}(\af)}\leq \fpt(\af)$ is only attained by height-1 ideals. In a forthcoming paper \cite{baily_classification_2026}, the author proves a stronger bound for higher-height ideals, analogous to a result of Demailly and Pham on log canonical thresholds \cite{demailly_sharp_2014}. Specifically, for an $F$-finite regular local ring $(R, \mf)$ of characteristic $p > 0$ and an ideal $\af\subseteq R$ of height $l$, there is a log convex sequence of positive integers $\sigma_0(\af),\dots, \sigma_l(\af)$ with $\sigma_0(\af) = 1, \sigma_1(\af) = \ord_{\mf}(\af)$, and
\begin{equation}\label{eqn:sigma-lowerbound}
\frac{\sigma_0(\af)}{\sigma_1(\af)} + \dots + \frac{\sigma_{l-1}(\af)}{\sigma_l(\af)} \leq \fpt(\af).
\end{equation}
The integers $\sigma_j(\af)$ are constructed in \cite{bivia-ausina_joint_2008}, where they are written as $\sigma_j(\af, \mf)$.
The main result of \cite{baily_classification_2026} is a description of the cases of equality in \Cref{eqn:sigma-lowerbound} when $R$ is a polynomial ring and $\af\subseteq R$ is a homogeneous ideal.

\begin{notation}
    In this paper, all rings are commutative and Noetherian. The letter $p$ always denotes a prime number and $q$ a power of $p$.
\end{notation}

\section{Preliminaries}
We consider with monomial orders and leading terms in power series rings. With a few exceptions, Gr\"obner theory in power series rings is analogous to the theory in polynomial rings. For the unfamiliar reader, we recommend \cite[Chapter 4] {cox_usingAG_2005}.
\begin{notation}
    For a vector of ring elements $\mathbf{f} = f_1,\dots, f_r$ and a vector of nonnegative integers $\mathbf{a} = a_1,\dots, a_r$, we let $\mathbf{f}^{\ab}$ denote the element $f_1^{a_1}\dots f_r^{a_r}$. If $\ab$ is instead a tuple of nonnegative real numbers, we let $\mathbf{f}^{\ab}$ denote the $\R$-divisor $a_1\text{div}(f_1) + \dots + a_r\text{div}(f_r)$. 
\end{notation}
\begin{defn}
    Let $k$ be a field. Let $R = k\llbracket x_1,\dots, x_n\rrbracket$. A \textit{local monomial order} on $R$ is a partial ordering $>$ on the set of monomials of $R$ such that for all $\ab, \bb, \cb$ in $\Z_{\geq 0}^n$, if $\mathbf{x}^{\ab} \leq \mathbf{x}^{\bb}$, then 
        \[
        \mathbf{x}^{\ab + \cb} \leq \mathbf{x}^{\bb + \cb} \leq \mathbf{x}^{\bb}.
        \]
    In particular, the greatest monomial in $R$ is $1$. 

    Let $f = \sum_{\ab\in \Z_{\geq 0}^r} \gamma_{\ab}\mathbf{x}^{\ab}$ be an element of $R$. Among all $\bb\in \Z^n_{\geq 0}$ such that $\gamma_\bb\neq 0$, let $S$ denote the set of vectors for which $\mathbf{x}^{\bb}$ is maximal with respect to $>$. The \textit{initial term} of $f$ with respect to $>$, denoted $\ini_>(f)$, is the element $\sum_{\bb\in S}\gamma_{\bb}\mathbf{x}^{\bb}$. 
\end{defn}
\subsection{Background on F-Thresholds and Test Ideals}
To begin, we define $F$-thresholds and test ideals and collect a few relevant properties. For further background, we refer the reader to \cite{mustata_f-thresholds_2004,takagi_f-pure_2004}. 

\begin{defn-prop}[\cite{mustata_f-thresholds_2004}]\label{defn-prop:f-threshold}
    Let $R$ be a regular local ring of characteristic $p > 0$. Let $\af\subseteq R$ be an ideal and $\bfr\subseteq R$ a proper ideal such that $\af\subseteq \sqrt{\bfr}$. For $e\geq 0$, define $\nu_{\af}^{\bfr}(p^e)$ to be the greatest integer $t$ such that $\af^t\not\subseteq \bfr^{[p^e]}$. Then the sequence $\dfrac{\nu_{\af}^{\bfr}(p^e)}{p^e}$ has a limit as $e\to \infty$, and we refer to this limit as the $F$-threshold $\ft^{\bfr}(\af)$. Additionally, we have
    \[
    \ft^{\bfr}(f) = \inf\left\{\frac{t}{p^e}:\af^t\subseteq \bfr^{[p^e]}\right\} = \sup\left\{\frac{t}{p^e}:\af^t\not\subseteq\bfr^{[p^e]}\right\}.
    \]
\end{defn-prop}
For the majority of this paper, the above notion of the F-threshold suffices. In a few isolated instances, we need a more general notion of what it means for a pair $(R, f^t)$ to be $F$-pure. 
\begin{defn}[\cite{schwede_singularities_2024}, \S 1.1]
    Let $R$ be a ring of characteristic $p > 0$. For any $e > 0$, the module $F^e_*R$ has underlying Abelian group isomorphic to $R$ and an $R$-module action defined by restriction of scalars along the Frobenius map $F^e: R\to R$. Concretely, the elements of $F^e_*R$ are $\{F^e_*f: f\in R\}$, and $xF^e_*f = F_*^e (x^{p^e}f)$. We say that $R$ is $F$-finite if $F_*R$ is a finite $R$-module, in which case $F^e_*R$ is finite for all $e > 0$.
\end{defn}
\begin{defn}[\cite{schwede_centers_2010,takagi_f-pure_2004}]\label{defn:general-f-pairs}
    Let $R$ be an $F$-finite reduced ring of characteristic $p > 0$. Let $\af\subseteq R$ be an ideal and $t$ a nonnegative real number. We say that the pair $(R, \af^t)$ is \textit{sharply $F$-pure} if, for a single $e>0$ (equivalently, all integers $ne$ where $n\in \Z^+$), there exists $d\in \af^{\ceil{t(p^e-1)}}$ such that the map
        \[
    R\to F^e_*R \qquad 1\mapsto F^e_*d
        \]
        splits. 
        
        The $F$-pure threshold $\fpt(\af)$ of the pair $(R, \af)$ is defined to be the supremum of all real $t \geq 0$ such that $(R, \af^t)$ is sharply $F$-pure.
\end{defn}
By \cite[Remark 1.5]{mustata_f-thresholds_2004}, if $(R, \mf)$ is a regular local ring then $\fpt(\af) = \ft^{\mf}(\af)$.
\begin{defn}[\cite{blickle_discreteness_2008,hara_generalization_2003}]\label{defn:test-ideal}
    Let $R$ be an $F$-finite regular ring of characteristic $p>0$, let $\af\subseteq R$ be an ideal, and let $t$ be a nonnegative real number. The \textit{test ideal} of the pair $(R, \af^t)$ is equal to the image, for $e\gg 0$, of the map
    \[
    \af^{\ceil{tp^e}}\cdot \Hom_R(F^e_*R, R)\to R \qquad d\cdot \phi\mapsto \phi(F^e_*d).
    \]
    We denote the test ideal of $(R, \af^t)$ by $\tau(R, \af^t)$ or by $\tau(\af^t)$ if the ambient ring is clear from the context.
\end{defn}
\begin{prop}\label{prop:properties-of-ft}
    Let $(R, \mf)$ be a regular  local ring of dimension $n$ and characteristic $p > 0$. Let $\af$ be an ideal of $R$ and let $\bfr\subseteq R$ such that $\af\subseteq \sqrt{\bfr}$. Then the following hold.
    \begin{enumerate} 
        \item If $\cf\subseteq \af$, then $\ft^{\bfr}(\cf)\leq \ft^{\bfr}(\af)$.
        \item For $s$ in $\Z^+$, we have $\ft^\bfr(\af^s) = \frac{1}{s}\ft^\bfr(\af)$.
        \item The inequalities $\dfrac{1}{\ord_\mf(f)}\leq \ft^{\mf}(f)\leq \dfrac{n}{\ord_\mf(f)}$ hold.
        \item For all $e > 0$, we have $\dfrac{\nu_{\af}^{\bfr}(p^e)}{p^e} < \ft^{\bfr}(\af)$.
        \item For $x$ in $\mf\setminus \mf^2$, if $\bar \af, \bar \bfr$ denote the images of $\af, \bfr$ in $R/xR$, then $\ft^\bfr(\af)\geq \ft^{\bar \bfr}(\bar \af)$.
        \item We have $\ft^{\bfr}(\af) = \inf \{t: \tau(\af^t)\subseteq \bfr\}$.
        \item If $R$ is a power series ring, $\bfr$ is a monomial ideal, and $>$ is a local monomial order, then $\ft^\bfr(\ini_>(\af)) \leq \ft^\bfr(\af)$.
    \end{enumerate}
\end{prop}
\begin{proof} For (1)-(4), see \cite[Proposition 1.7]{mustata_f-thresholds_2004}. The proof of claims (5)-(7) are enumerated below.
    \begin{enumerate}
        \setcounter{enumi}{4}
        \item This standard fact follows from an observation in \cite[Theorem 3.11]{takagi_f-singularities_2004}: for any positive integers $t,e$, if $\af^t\subseteq \bfr^{[p^e]}$ then $(\bar \af)^{t} \subseteq (\bar \bfr)^{[p^e]}$. The assumption that $x\in \mf\setminus \mf^2$ is to ensure that the quotient ring $R/xR$ is again regular, but is not necessary if one defines the $F$-threshold in greater generality.
        \item This fact follows from \cite[Proposition 2.7]{mustata_f-thresholds_2004}.
        \item The argument of the claim in \cite[Proof of Proposition 4.5]{takagi_f-pure_2004} works verbatim.
    \end{enumerate}
\end{proof}

\subsection{Critical Points, after Hern\'andez and Teixeira}
\begin{notation}
    For a vector $\ab = (a_1,\dots, a_n)$ in $\R^n_{\geq 0}$, the symbol $\norm{\ab}$ denotes the number $a_1 + \dots + a_n$.
\end{notation}
A key point in the proof of \Cref{thm:E1-char-p} uses ideas of Hern\'andez and Teixera \cite{hernandez_syzygy_2017} in the 2-dimensional setting, generalized to fit our needs. In op. cit., the authors consider the following problem.
\begin{question}\label{question:linear_fractal}
    Let $k$ be a field of characteristic $p > 0$ and $\hb =h_1,\dots, h_r$ be a tuple of distinct linear forms in $R =k[x, y]$. For which $\tb$ in $\R_{\geq 0}^r$ is $\tau(R, \hb^{\tb}) = R$? More generally, for a homogeneous system of parameters $U, V$ and $\bfr = (U, V)$, how can we describe the set of $\tb$ in $\R_{\geq 0}^r$ for which $\tau(R, \hb^{\tb})\subseteq \bfr$?
\end{question}
To address this question, the authors study \textit{syzygy gaps.} 
For $\ab$ in $\Z_{\geq 0}^r$ and $q = p^e$, one can determine whether $\tau(R, \hb^{\ab/q})\subseteq \bfr$ by computing the graded free resolution 
\[
R(-m)\oplus R(-n)\to R(-q\deg U)\oplus R(-q\deg V)\oplus R(-\norm{\ab})\to R/(U,V,\hb^{\ab})\to 0.
\]
The number $\Delta(\ab)$ is defined to be the ``syzygy gap'' $|m-n|$. Defining $\Delta(\ab/q)$ to be $\frac{1}{q}\Delta(\ab)$ yields a well-defined map $\Delta: \left(\Z[\frac{1}{p}]_{\geq 0}\right)^r\to \Z[\frac{1}{p}]_{\geq 0}$. By Lemma 3.2 and Proposition 3.5 in op. cit., $\Delta$ extends uniquely to a continuous map from $\R_{\geq 0}^r$ to $ \R_{\geq 0}$ and determines the behavior of the test ideal: for $\tb$ in $\R_{\geq 0}^r$, we have $\tau(R, \hb^{\tb})\subseteq \bfr$ if and only if $\Delta(\tb) = |\norm{\tb} - \deg(UV)|$. By Corollary 3.11 in op. cit., we have $\tau(R, \hb^{\tb})\subseteq \bfr$ whenever $\norm{\tb}\geq \deg(UV)$, so the nontrivial behavior of the test ideal is confined to the region $\norm{\tb} < \deg(UV)$. By Theorem 5.9 in op. cit., this nontrivial behavior is completely determined by a family of distinguished points in $\Z[\frac{1}{p}]_{\geq 0}$ that the authors refer to as \textit{critical points}.

The question we consider is similar. Rather than arbitrary parameter ideals, we consider special parameter ideals $\bfr$ of the form $(y, x^\ell)$ and rather than tuples of linear forms, we consider tuples of polynomials $\hb = h_1,\dots, h_r$, where $h_i = y - g_i$ and $g_i\in x^\ell k[x]$. 
\begin{question}\label{question:polynomial_fractal}
    Let $k$ be a field of characteristic $p > 0$ and $R = k\llbracket x,y\rrbracket$. Let $\ell > 0$ and let $\bfr:= (y, x^\ell)$. Let $r\geq 2$ and for $1\leq i\leq r$, let $h_i = y - g_i$, where $g_i\in x^\ell k[x]$. Let $\hb$ denote the vector $h_1,\dots, h_r$. For which $\tb$ in $\R_{\geq 0}^r$ is $\tau(R, \hb^{\tb})\subseteq \bfr$?
\end{question}
Analogously to \Cref{question:linear_fractal}, we have $\tau(R, \hb^{\tb})\not\subseteq \bfr$ when $\norm{\tb} < 1$ and $\tau(R, \hb^{\tb}) \subseteq \bfr$ when $\norm{\tb}\geq 2$, so the nontrivial behavior of $\tau(R,\hb^{\tb}) $ is contained in the strip $1\leq \norm{\tb} < 2$. Without access to an $\N$-grading, we cannot construct an analog of $\Delta$, but we prove that an analog of critical points do exist in the setting of \Cref{question:polynomial_fractal}.

If $D$ is the largest degree of any of the polynomials $g_i$, then the behavior of $\tau(R, \hb^{\tb})$ in the strip $1\leq \norm{\tb} < 1 + \frac{\ell}{D}$ is determined by critical points in $\left(\Z[\frac{1}{p}]_{\geq 0}\right)^r$ by \Cref{cor:critical-point-computes-ft}, whereas for $1 + \frac{\ell}{D} \leq \norm{\tb} < 2$, critical points do not always determine $\tau(R, \hb^{\tb})$; see \Cref{example:critical-points-insufficient}.
\begin{defn}\label{defn:regions}
     Consider the setup of \Cref{question:polynomial_fractal}. We define a partition of the set $\left(\Z[\frac{1}{p}]_{\geq 0}\right)^r$ into the \textit{upper and lower regions attached to $\hb$ and $\bfr$}. For $\ab$ in $ \Z_{\geq 0}^r, q = p^e$, we say that $\frac{\ab}{q}\in \mathscr U$ (the upper region) if $\hb^\ab = h_1^{a_1}\dots h_r^{a_r}\in \bfr^{[q]}$. Otherwise, $\frac{\ab}{q}\in \mathscr L$ (the lower region).
\end{defn}
Because $R$ is $F$-split, it follows that $\hb^\ab \in \bfr^{[q]}$ if and only if $\hb^{p\ab}\in \bfr^{[pq]}$, so the regions $\mathscr U,\mathscr L$ are well-defined. 
\begin{lemma}\label{lemma:U-pward-closed}
    For any $\frac{\ab}{q} \leq \frac{\ab'}{q'}$, if $\frac{\ab}{q}\in \mathscr U$ then $\frac{\ab'}{q'}\in\mathscr U$. 
\end{lemma}
\begin{proof}
    By well-definedness of $\mathscr U$, we may set $q'' = \max(q,q')$ and rewrite $\frac{\ab}{q} = \frac{\bb}{q''}, \frac{\ab'}{q'} = \frac{\cb}{q''}$. Suppose $\frac{\bb}{q''}\in \mathscr U$. As $\bb\leq \cb$, it follows that $\hb^\bb\mid \hb^\cb$, so $\hb^\cb\in \bfr^{[q'']}$ and thus $\frac{\cb}{q''}\in \mathscr U$.
\end{proof}
\begin{lemma}\label{lemma:e_i-in-U}
Suppose that $\ab\in \Z_{\geq 0}^r$ and that $q$ is a power of $p$. Let $e_1,\dots, e_r$ denote the standard unit vectors in $\Z^r$. For any $1\leq s\leq r$, we have $\frac{\ab}{q} - \frac{e_s}{q}\in \mathscr U$ if and only if $\frac{\ab}{q} - \frac{e_s}{pq}\in \mathscr U$.
\end{lemma}
\begin{proof}
The implication $\Rightarrow$ follows from \Cref{lemma:U-pward-closed}.  For the reverse implication, start by considering the map $\phi: R\to R$ given by $\phi(x) = x$ and $\phi(y) = h_s$. Since $\phi$ is an isomorphism and $\phi(\bfr) = \bfr$, we may without loss of generality assume $h_s = y$. 

Suppose $\frac{\ab}{q} - \frac{e_s}{pq}\in \mathscr U$. Then $\hb^{p(\ab - e_s)}y^{p-1} = \hb^{p\ab - e_s} \in \bfr^{[pq]}$. Applying \cite[Lemma 3.5]{kadyrsizova_lower_2022}, we obtain
\[\hb^{p(\ab-e_s)} \in (\bfr^{[pq]}:y^{p-1}) = (y^{pq-p+1},x^{\ell pq})\]
The monomials $\{x^iy^j\}_{0\leq i,j \leq p-1}$ form a free basis for $F_*R$ over $R$; write $\Phi: F_*R\to R$ for projection onto the $F_*(xy)^{p-1}$ factor. Then $\Phi\circ (\cdot F_*(xy)^{p-1})$ is a splitting of the Frobenius map $R\to F_*R$, so 
\[
\hb^{\ab-e_s} = \Phi(F_*(xy)^{p-1}\hb^{p(\ab-e_S)}) \in\Phi(F_*(y^{pq}, x^{\ell pq +p-1})) = \bfr^{[q]}.
\]
\end{proof}
\begin{lemma}\label{lemma:crit-point}
    Let $\frac{\ab}{q}$ be a vector in $\mathscr U$. Then $\frac{\ab - e_s}{q}\in \mathscr L$ for all $1\leq s\leq r$ such that $a_s > 0$ if and only if $\frac{\ab'}{q'}\in \mathscr L$ for all $\frac{\ab'}{q'} < \frac{\ab}{q}$.
\end{lemma}
\begin{proof}
    The implication $\Leftarrow$ is tautological. For the implication $\Rightarrow$, suppose that $\frac{\ab - e_s}{q}\in \mathscr L$ for all $1\leq s\leq r$ with $a_s > 0$. By \Cref{lemma:e_i-in-U}, it follows that $\frac{p^n\ab-e_s}{p^nq}\in \mathscr L$ for all $n > 0$ and all $s$ with $a_s > 0$. For $\frac{\ab'}{q'} < \frac{\ab}{q}$, suppose without loss of generality that $\ab'_1 < \ab_1$. Then $\frac{\ab'}{q'} \leq \frac{q'\ab - e_1}{qq'}$, which implies $\frac{\ab'}{q'}\in \mathscr L$ by \Cref{lemma:U-pward-closed}.
\end{proof}
\begin{defn}[c.f. \cite{hernandez_f-threshold_2017}, Definition 5.1]\label{defn:crit-point}
Let $\frac{\cb}{q}$ be a vector in $\mathscr U$. If $\frac{\cb}{q}$ satisfies either of the equivalent conditions of \Cref{lemma:crit-point}, we say that $\frac{\cb}{q}$ is a \textit{critical point} attached to $\hb$ and $\bfr$. We let $\mathscr C$ denote the set of such critical points.
\end{defn}
\begin{lemma}[c.f. \cite{hernandez_syzygy_2017}, Corollary 5.7]\label{lemma:critical-point-below}
    Let $a$ be in $\Z_{\geq 0}^r$ and $q$ be a power of $p$. Then $\frac{\ab}{q}\in \mathscr U$ if and only there exists a critical point $\frac{\cb}{q}$ in $\mathscr C$ with $\cb\leq \ab$.
\end{lemma}
\begin{proof}
    Suppose $\frac{\ab}{q}\in \mathscr U$. The set $S$ of elements $\ub$ in $\Z^r_{\geq 0}$ such that $\frac{\ub}{q}\in \mathscr U$ and $ \ub\leq \ab$ is finite and nonempty, so choose $\cb\in S$ such that $\norm{\cb}$ is minimal. By construction, we have $\frac{\cb-e_s}{q}\in \mathscr L$ for all $1 \leq s\leq r$, hence $\frac{\cb}{q}$ is a critical point by \Cref{lemma:crit-point}. Conversely, if $\frac{\ab}{q}\geq \frac{\cb}{q}\in \mathscr C$, then $\frac{\ab}{q}\in \mathscr U$ by \Cref{lemma:U-pward-closed}.
\end{proof}
\begin{lemma}\label{lemma:norm-of-upper-region}
    If $\frac{\cb}{q}\in \mathscr U$, then $\norm{\frac{\cb}{q}} \geq 1$.
\end{lemma}
\begin{proof}
First, we note that $h_i\in \bfr$ for all $1\leq i\leq r$, so $\hb^\ab\in \bfr^{\norm{\ab}}$ for any $\ab$ in $\Z^r_{\geq 0}$. Conversely, the coefficient of $y^{\norm{\ab}}$ in $\hb^\ab$ is equal to 1, so $\ord_{\bfr}(\hb^{\ab}) = \norm{\ab}$. If $\frac{\cb}{q}\in \mathscr U$, then $\hb^\cb\in \bfr^{[q]}\subseteq \bfr^q$, so $\norm{\cb} = \ord_\bfr(\hb^\cb) \geq q$.    
\end{proof}
\begin{lemma}[c.f. \cite{hernandez_syzygy_2017}, Remark 5.8]\label{lemma:unique-critical-point}
   Let $D$ denote the largest total degree of any of the polynomials $g_1,\dots, g_r$. If $\frac{\bb}{q}, \frac{\cb}{q}$ are distinct elements of $\mathscr C$, then $\norm{\max(\frac{\bb}{q},\frac{\cb}{q})} \geq 1 +\frac{\ell}{D}$.
\end{lemma}
\begin{proof}

Set $\cb^{(0)} = \cb$. To prove the claim, we produce a sequence of critical points $\cb^{(0)},\cb^{(1)},\dots, \cb^{(s)}$ and possibly non-critical points $\ab^{(0)},\dots, \ab^{(s)}$ such that $\ab^{(s)}\leq \max(\bb, \cb^{(0)})$; we will show that $\norm{\ab^{(s)}}\geq  1 + \frac{\ell}{D}$.

Consider the following auxiliary conditions, which we will show hold for all $1\leq m\leq s$:
\begin{align}
\max(\bb,\cb^{(m)})\leq &\max(\bb, \ab^{(m-1)})\label{eqn:crit_pt_below_a}\\
 &\max(\bb, \ab^{(m)}) < \max(\bb, \cb^{(m)})\label{eqn:strict_descent}
\end{align} 

Suppose we have constructed $\cb^{(0)},\dots, \cb^{(n)}, \ab^{(0)},\dots, \ab^{(n-1)}$ satisfying the conditions
\begin{enumerate}[(i)]
\item $\cb^{(n)}\neq \bb$;
\item $\frac{\ab^{(0)}}{q},\dots, \frac{\ab^{(n-1)}}{q}\in \mathscr U$;
\item Equation \ref{eqn:crit_pt_below_a} holds for $1\leq m\leq n$;
\item Equation \ref{eqn:strict_descent} holds for $0\leq m\leq n-1$.
\end{enumerate}
For the base case $n = 0$, the only non-vacuous hypothesis is that $\cb^{(0)}\neq \bb$, which holds by assumption. By \Cref{defn:crit-point}, the distinct critical points $\frac{\bb}{q},\frac{\cb^{(n)}}{q}$ are incomparable, so let $1\leq i_n, j_n\leq r$ such that $c^{(n)}_{i_n} > b_{i_n}$ and $c^{(n)}_{j_n} < b_{j_n}$. Set $\ab^{(n)} = \max(\bb, \cb^{(n)}) - e_{i_n} - e_{j_n}$. We compute 
\begin{equation}\label{eqn:max_of_a_and_b}
\max(b_t, a_t^{(n)}) = \begin{cases}
    c^{(n)}_{i_n} - 1 & t={i_n}\\
    b_{j_n} & t = {j_n}\\
    \max(c^{(n)}_t,b_t) & t\neq i_n,j_n,
\end{cases}
\end{equation}
so Equation \ref{eqn:strict_descent} holds for $m = n$. If $\frac{\ab^{(n)}}{q}\in \mathscr L$, we set $s = n$ and terminate the sequence. Otherwise, we apply \Cref{lemma:critical-point-below} to produce $\cb^{(n+1)}\leq \ab^{(n)}$ where $\frac{\cb^{(n+1)}}{q}\in \mathscr C$. By construction, Equation \ref{eqn:crit_pt_below_a}
 holds for $m=n+1$. As $c^{(n+1)}_{j_n} \leq a^{(n)}_{j_n} = b_{j_n} - 1$, the vectors $\cb^{(n+1)}, \bb$ are again distinct. 

By Equations \ref{eqn:crit_pt_below_a},\ref{eqn:strict_descent}, we have a strictly decreasing sequence of nonnegative integer vectors
\[
\max(\bb, \cb^{(0)}) > \max(\bb, \cb^{(1)}) > \dots 
\]
which necessarily terminates. In particular, there exists some index $s\geq 0$ such that $\frac{\ab^{(s)}}{q}\in \mathscr L$. Additionally, by construction of $\ab^{(s)}$ we compute
\[
\ab^{(s)} = \max(\bb, \cb^{(s)}) - e_{i_s}-e_{j_s} \leq \max(\bb, \cb^{(0)}) - e_{i_s} - e_{j_s},
\]
so $\norm{\max(\bb, \cb^{(0)})}\geq \norm{\ab^{(s)}}+2$. To prove that $\norm{\max(\frac{\bb}{q}, \frac{\cb^{(0)}}{q})}\geq 1 + \frac{\ell}{D}$, it therefore suffices to show that 
\begin{equation}\label{eqn:bound-from-total-degree}
\norm{\ab^{(s)}}\geq q + \frac{\ell q}{D} - 2
\end{equation}



Without loss of generality, write $\ab = \ab^{(s)}, \cb = \cb^{(s)}, i_s = 1$ and $j_s = 2$.  With $h_1 = y - g_1, h_2 = y- g_2$, consider the automorphism $\phi: R\to R$ given by $x\mapsto x$ and $y\mapsto y - g_1$. Then $\phi(h_1) = y, \phi(\bfr) = \bfr$, and $\phi(h_i) = y - g_i - g_1$, where $g_i - g_1$ is a polynomial in $k[x]$ of total degree at most $D$. We may therefore assume without loss of generality that $g_1 = 0$ so that $h_1 = y$.

Set $m = \ord_x(g_2)$. As $x^\ell$ divides $g_2$ and $\deg(g_2)\leq D$, we necessarily have $\ell \leq m\leq D$. We compute $(h_1, h_2) = (y, x^m)$. As $\frac{\ab}{q}\in \mathscr L,$ we have ${\hb}^{\ab}\notin \bfr^{[q]}.$ On the other hand, $\frac{\ab+e_1}{q}\geq \frac{\bb}{q}$ and $\frac{\ab+e_2}{q}\geq \frac{\cb}{q}$, so $\frac{\ab+e_1}{q},\frac{\ab+e_2}{q}\in \mathscr U$ by \Cref{lemma:critical-point-below}; it follows that $\hb^{\ab}(h_1,h_2)\subseteq \bfr^{[q]}$. Rewriting $(h_1,h_2)$ as $(y, x^m)$ and applying \cite[Lemma 3.5]{kadyrsizova_lower_2022}, we deduce that 
\[
\hb^{\ab}\in (\bfr^{[q]}: (h_1,h_2)) = (y^q,x^{\ell q},  x^{\ell q-m}y^{q-1}).
\]
Write $\hb^{\ab} = \sum_{i\geq 0}z_iy^i$, where $z_i\in k[x]$. The quotient module \[\dfrac{(y^q, x^{\ell q}, x^{\ell q - m}y^q-1)}{(y^q, x^{\ell q})}\] is spanned by the monomials $y^{q-1}x^{\ell q - m},\dots, y^{q-1}x^{\ell q - 1}$, so $z_{q-1}$ is nonzero and has degree at least $lq-m$. On the other hand, $z_{q-1}$ is a polynomial of degree $\norm{\ab}-(q-1)$ in the inputs $g_1,\dots, g_r$, so we have
\[
\ell q - m \leq \deg(z_{q-1}) \leq D(\norm{\ab} - (q-1)).
\]
As $m \leq D$, we conclude that $\norm{\ab}\geq (q-1) + \frac{\ell q - D}{D}$, so \Cref{eqn:bound-from-total-degree} holds.
\end{proof}
\begin{cor}[c.f. \cite{hernandez_syzygy_2017}, Theorem 5.9]\label{cor:critical-point-computes-ft}
    Let $h_1,\dots, h_r$ be polynomials as in \Cref{question:polynomial_fractal} where the $g_i$ are distinct. Let $D$ denote the largest degree of the polynomials $g_1,\dots, g_r$. Let $f$ be an element of $k\llbracket x,y\rrbracket$
    such that $f = \hb^\tb = h_1^{t_1}\dots h_r^{t_r}$ for some positive integers $t_1,\dots, t_r$. Setting $\bfr = (y, x^\ell)$, either $\ft^{\bfr}(f) \geq \frac{1 + \ell/D}{\norm{\tb}}$ or there exists a unique critical point $\cb\leq \frac{1 + \ell/D}{\norm{\tb}}\tb$ which computes $\ft^{\bfr}(f)$: that is, $\ft^{\bfr}(f) = \max_{1\leq i\leq r} \frac{c_i}{t_i}$.
\end{cor}
\begin{proof}
    Let $\lambda = \frac{1 + \ell/D}{\norm{\tb}}$. Write $\mu:=\ft^{\bfr}(f)$ and suppose that $\mu < \lambda$. For all $e > 0$, \Cref{prop:properties-of-ft} (4) implies that
    $f^{\ceil{p^e\mu}}\in \bfr^{[p^e]}$. By definition, $\frac{\ceil{p^e\mu}}{p^e}\tb\in \mathscr U$, so by \Cref{lemma:critical-point-below}, there exists $\cb^{(e)}$ in $\mathscr C$ with $\cb^{(e)}\leq \frac{\ceil{p^e\mu}}{p^e}\tb$. Choose $e_0 \gg 0$ such that for all $e \geq e_0$ we have $\frac{\ceil{p^e\mu}}{p^e} < \lambda$. For any $e > e_0$ we have \[
    \norm{\max(\cb^{(e_0)}, \cb^{(e)})} \leq \norm{\frac{\ceil{p^e\mu}}{p^e}\tb} = \frac{\ceil{p^e\mu}}{p^e}\norm{\tb} < \lambda\norm{\tb} = 1 + \frac{\ell}{D}.\] By \Cref{lemma:unique-critical-point}, we must have $\cb^{(e)} = \cb^{(e_0)}$. Let $\cb:= \cb^{(e_0)}$.
    
    As $\cb\leq \frac{\ceil{p^e}\mu}{p^e}\tb$ for all $e \geq e_0$, it follows that $\cb\leq \mu \tb$. On the other hand, for any $\mu' < \mu$, we may choose $e\gg 0$ such that $\ceil{p^e\mu'} < \nu_f^{\bfr}(p^e)$ by \Cref{defn-prop:f-threshold}. Consequently, $\frac{\ceil{p^e\mu'}}{p^e}\tb \notin \mathscr U$, so $\cb^{(e_0)}\not\leq \frac{\ceil{p^e\mu'}}{p^e}\tb$ and hence $\cb\not\leq \mu'\tb$. We conclude that $\mu$ is the smallest real number for which $\cb\leq \mu \tb$, hence $\mu = \max_{1\leq i\leq r} \frac{c_i}{t_i}$.
\end{proof}
\begin{remark}
Putting $\ell = n = 1$, the above corollary gives an alternate proof of \cite[Theorem 5.9]{hernandez_syzygy_2017} in the special case $\bfr = (x,y)$: given any homogeneous polynomial $f$, we may apply a linear change of coordinates so that $x\nmid f$, after which \Cref{cor:critical-point-computes-ft} applies. 
\end{remark}
The following example shows that the parameter $D$ is necessary. Unlike the homogeneous case (\cite[Theorem 5.9]{hernandez_f-threshold_2017}), if $\ft^{\bfr}(f) < \frac{2}{\norm{\tb}}$, there may not be a unique critical point $\frac{\cb}{q}\leq \frac{2\tb}{\norm{\tb}}$ computing $\ft^{\bfr}(f)$.
\begin{example}\label{example:critical-points-insufficient}
    Let $R = \F_2\llbracket x, y\rrbracket$ and $\ell = 1$ so that $\bfr = (x,y)$ and $\ft^{\bfr}(-) = \fpt(-)$. We define $h_1 = y + x, h_2=y + x^2, h_3 = y+x^4$. We consider $\tb = (1,2,1), f:= \hb^{\tb} = h_1h_2^2h_3$. Then $f^7\in \mf^{[16]}$, so $\fpt(f) \leq \frac{7}{16} < \frac{2}{\norm{\tb}}$. There are many critical points below $\frac{2\tb}{\norm{\tb}} = (\frac{1}{2}, 1, \frac{1}{2})$, however: for instance, the points $\left(\dfrac{1}{2}, 1, \dfrac{1}{2}\right), \left(\dfrac{3}{8}, \dfrac{13}{16}, \dfrac{7}{16}\right), \left(\dfrac{1}{4}, \dfrac{7}{8}, \dfrac{3}{8}\right)$ are all critical. Moreover, none of these points compute the actual value of $\fpt(f)$, which is equal to $\dfrac{3}{7}$ by a computation in Macaulay2 \cite{M2} using the FrobeniusThresholds package \cite{FrobeniusThresholdsSource} .
\end{example}
\section{Proof of \texorpdfstring{\Cref{thm:E1-char-p}}{Theorem 1.1}}\label{sec:proof}
We state a general version of the main theorem and outline a proof.
\begin{theorem}\label{thm:E1-char-p}
Let $(R, \mf)$ be a regular local ring of characteristic $p > 0$ and $\af\subseteq R$ a proper ideal. Let $d = \ord_\mf(\af)$ and write $d = qs$ where $q$ is a power of $p$ and $\gcd(p,s)=1$. Then 
\begin{equation}\label{eqn:E1-char-p-hard}
   \text{If $\ft^{\mf}(\af) = \frac{1}{d}$, then there exists $g$ in $ \mf^{[q]}\widehat{R}$ such that $g^s \widehat{R} = \af \widehat{R}$.} 
\end{equation}
Suppose further that, for all prime elements $\pi$ in $ R$, the formal fiber $\widehat{R}\otimes_R R_{(\pi)}/\pi R_{(\pi)}$ is reduced. Then 
\begin{equation}\label{eqn:E1-char-p-excellent}
    \text{If $\ft^{\mf}(\af) = \frac{1}{d}$, there exists $h$ in $ \mf^{[q]}$ such that $h^s R = \af R$.}
\end{equation}
\end{theorem}
Since $F$-finite rings are excellent \cite[Remark 13.6]{gabber_notes_2004}, our main theorem follows from this more precise version. A version for pairs $(R, \af_1^{t_1}\dots \af^{t_r}_{r})$ where $t_i\in \R$ will appear in the author's dissertation.

As a special case of \Cref{thm:E1-char-p}, one can classify homogeneous polynomials $f$ in $ k\llbracket x_1,\dots, x_n\rrbracket$ for which $\fpt(f) = \dfrac{1}{\deg(f)}$. This case is given in \cite[Remark 3.2]{kadyrsizova_lower_2022}; our proof follows a similar strategy. We reduce the claim to the case of a principal ideal in a complete local ring over an algebraically closed field, after which we apply a local Bertini theorem (\Cref{lemma:reduced-bertini-flenner}) to reduce to the 2-dimensional case. 

At this point, the two arguments diverge: unlike homogeneous polynomials, a power series in two variables may be irreducible, so we use the Weierstrass preparation theorem to write $f$ as $u(y^d + a_1y^{d-1}+\dots + a_d)$ where $u\in k\llbracket x,y\rrbracket^\times$ and  $a_1,\dots, a_d\in k\llbracket x\rrbracket$ (\Cref{lemma:weierstrass-preparation}). Next, we pass to a finite flat extension $(R,\mf) \to (U, \nf)$ over which $f$ factors as $(y-\theta_1)\dots (y-\theta_d)$. This comes at the cost of having to consider the $F$-threshold of $f$ at $\mf U$ instead of the $F$-pure threshold of $f$. Using \cite[Proposition 5.1]{sato_accumulation_2023}, we reduce to the case that the $\theta_i$ are polynomials in a finite extension $k[x^{1/u}]$ of $k[x]$. Finally, we use the critical point framework (\Cref{cor:critical-point-computes-ft}) to deduce that $\theta_1=\dots =\theta_d$, after which we apply \Cref{lemma:s-root} to deduce that $f$ admits an $s$th root in $R$. As $d = qs$, if $f = g^s$, then $\fpt(g) = s\fpt(f) = \dfrac{1}{q} = \dfrac{1}{\ord_{\mf}(g)}$, so the result follows from the degree-$q$ case (\Cref{lemma:degree-q-case}). We divide the proof into five subsections.
\begin{enumerate}
    \item \Cref{eqn:E1-char-p-hard} holds when $\af$ is principal, $k = \bar{k}$, and $R = k\llbracket x, y\rrbracket$.
    \item \Cref{eqn:E1-char-p-hard} holds when $\af$ is principal, $k = \bar{k}$, and $R = k\llbracket x_1,\dots, x_n\rrbracket$.
    \item \Cref{eqn:E1-char-p-hard} holds when $\af$ is principal, $k$ is any field, and $R = k\llbracket x_1,\dots, x_n\rrbracket$.
    \item \Cref{eqn:E1-char-p-hard} holds when $\af$ is any ideal, $k$ is any field, and $R = k\llbracket x_1,\dots, x_n\rrbracket$.
    \item \Cref{thm:E1-char-p} holds.
\end{enumerate}
\subsection{Step (1)}
\begin{lemma}\label{lemma:degree-q-case}
    Let $(R, \mf)$ be a regular local ring of characteristic $p > 0$ and let $f$ be an element of $R$. If $q = p^e$ and $\ord_\mf(f) =q$, then $\ft^{\mf}(f) = \frac{1}{q}$ if and only if $f\in \mf^{[q]}$. 
\end{lemma}
\begin{proof}
    If $f^1\in \mf^{[q]}$, the inequality $\ft^{\mf}(f)\leq \frac{1}{q}$ follows from \Cref{defn-prop:f-threshold}. Conversely, if $\ft^{\mf}(f) = \frac{1}{q}$, then $\dfrac{\nu_f^{\mf}(q)}{q} < \frac{1}{q}$, so $\nu_f^{\mf}(q) = 0$ and $f\in \mf^{[q]}$.
\end{proof}
\begin{lemma}[\cite{stacks-project}\href{https://stacks.math.columbia.edu/tag/05CK}{Tag 05CK}]\label{lemma:intersection-pure}
    Let $A\to B$ be a faithfully flat map of rings and $I\subseteq A$ an ideal. Then $IB\cap A = I$.
\end{lemma}
\begin{lemma}\label{lemma:flat-ft}
    Let $(R,\mf)\to (S,\nf)$ be a flat local map of regular local rings. Suppose $\af,\bfr\subseteq R$ are ideals with $\af\subseteq \sqrt{\bfr}$. Then $\ft^{\bfr}(\af) = \ft^{\bfr S}(\af S)$ and $\ord_{\bfr}(\af) = \ord_{\bfr S}(\af S)$.
\end{lemma}
\begin{proof}
The first claim follows from \cite[Proposition 2.2 (v)]{huneke_closure_2008}. For the second, let $t > 0$. \Cref{lemma:intersection-pure} implies
    \[
    \af\subseteq \bfr^t\implies \af S\subseteq \bfr^t \implies \af S\cap R\subseteq \bfr^t\cap R \iff \af\subseteq \bfr^t,
    \]
    so $\ord_{\bfr}(\af) = \ord_{\bfr S}(\af S)$.
\end{proof}
\begin{lemma}\label{lemma:weierstrass-preparation}
    Let $k$ be an algebraically closed field of characteristic $p>0$. Let $R = k\llbracket z,w\rrbracket$ and set $\mf = (z,w)$. If $f\in \mf^d\setminus \mf^{d+1}$ and $\fpt(f) = \frac{1}{d}$, then there exists:
    \begin{itemize}
        \item A regular system of parameters $x, y$ for $R$;
        \item A degree-$d$ monic polynomial $P$ in $ k\llbracket x\rrbracket[y]$;
        \item A unit $u$ in $ R^\times$
    \end{itemize} 
    such that $f = uP$.
\end{lemma}
\begin{proof}
Write $f = f_d + f_{> d}$, where $f_d$ is homogeneous of degree $d$ and $f_{>d}\in \mf^{d+1}$. By \Cref{prop:properties-of-ft} (7), we have $\fpt(f_d) \leq \fpt(f)$. As the lower bound $\frac{1}{d} = \frac{1}{\ord_\mf(f_d)}
\leq \fpt(f_d)$ of \Cref{prop:properties-of-ft} (3) still applies, we have $\fpt(f_d) = \frac{1}{d}$. By \cite[Remark 3.2]{kadyrsizova_lower_2022}, we have $f_d = \ell^d$ for some homogeneous linear form $\ell$. Choosing a new regular system of parameters $x,y$ for $R$ such that $\ell = y$, we have $f = y^d + f_{>d}$. 

Write $f = a_0 + a_1y + \dots$ as a power series in $k\llbracket x\rrbracket\llbracket y\rrbracket$. As $f-y^d\in \mf^{d+1}$, it follows that $a_0,\dots, a_{d-1}\in \mf$ and $a_d\equiv 1 \mod\mf$. The claim now follows from the Weierstrass preparation theorem \cite[Theorems IV.9.1, IV.9.2]{lang_algebra_2002}.

\end{proof}
\begin{lemma}\label{lemma:s-root}
    Let $L$ be a field. Let $f$ be an element of $L\llbracket x\rrbracket$ with $x\nmid f$. Let $u,s$ be positive integers such that $s$ is not a multiple of $\char k$. If $g\in L\llbracket x^{1/u}\rrbracket$ such that $g^s = f$, then $g\in L\llbracket x\rrbracket$.
\end{lemma}
\begin{proof}
    Write $g = a_0 + a_1x^{1/u} + \dots \in L\llbracket x^{1/u}\rrbracket$. As $x\nmid f$, we have $a_0 \neq 0$. If $a_i = 0$ for all $u\nmid i$, then $g\in L\llbracket x\rrbracket$. Otherwise, for the sake of contradiction, let $i$ be minimal such that $u\nmid i$ and $a_i\neq 0$. By equating the coefficients of $x^{i/u}$ in $g^s$ and in $f$, we obtain
    \[
    0 = s a_0^{s-1}a_i.
    \]
    We assumed that $s$ is not a multiple of $\char k$, so $s$ is a unit in $k$, which contradicts the assumption that $a_0 a_i\neq 0$.
\end{proof}
\begin{lemma}\label{lemma:root-vanishing}
    Let $k$ be a field and $T = k\llbracket t\rrbracket$. Let $\theta_1,\dots, \theta_d$ be elements of $T$ and consider $f = (y-\theta_1)\dots (y-\theta_d)$ as an element of $T[y]$. Let $\ell\geq 1$. If $f\in (y, t^\ell)^d$, then $\ord_t(\theta_i)\geq \ell$ for all $1\leq i\leq d$. 
\end{lemma}
\begin{proof}
    Write $f = y^d + g_1y^{d-1} + \dots + g_d$ where $g_1,\dots, g_d\in T$. Re-order the roots $\theta_1,\dots, \theta_d$ such that 
    \[\ord_t(\theta_1) = \dots = \ord_t(\theta_r) < \ord_t(\theta_{r+1})\leq \dots \leq \ord_t(\theta_d).\]
    As $f\in (y, t^\ell)^d$, we have $y^{d-r}g_r \in (y, t^\ell)^d$ and hence $g_r\in (y, t^\ell)^{d-r}\cap T = t^{\ell r}T$. We have $g_r = \theta_1\dots \theta_r + \text{higher order terms}$, so $\ord_t(g_r) = \ord_t(\theta_1\dots\theta_r) = r\ord_t(\theta_1)$, so for all $1\leq i\leq d$ we have $\ord_t(\theta_i) \geq \ord_t(\theta_1) \geq \frac{\ell r}{r} = \ell$.
\end{proof}
\begin{lemma}\label{lemma:dim-2-alg-closed}
    Let $k$ be an algebraically closed field of characteristic $p > 0$. Let $R = k\llbracket x,y\rrbracket, \mf = (x,y)$. Let $d$ be a positive integer and $f$ an element of $ R$ such that $\ord_\mf(f) = d,\fpt(f) = \frac{1}{d}$. Write $d = qs$, where $q$ is a power of $p$ and $\gcd(p,s)=1$. Then there exists $g$ in $ \mf^{[q]}$ such that $g^s = f$.
\end{lemma}
\begin{proof}
By \Cref{lemma:weierstrass-preparation}, we may assume $f$ is a monic polynomial of degree $d$ in $k\llbracket x\rrbracket[y]$. Set $S = k\llbracket x\rrbracket$ and $K = \text{Frac}(S)$. Every monic polynomial in $S[y]$ factors completely over $\bar{K}$, so we may write 
\[
f = (y - \theta_1)^{e_1}\dots (y-\theta_r)^{e_r}
\]
for some $\theta_1,\dots, \theta_r$ distinct roots in $\bar{K}$; unlike the previous lemma, we will count roots with multiplicity.

Let $T$ denote the integral closure of $S$ in $K(\theta_1,\dots, \theta_r)$. By \cite[Theorem 4.3.4]{huneke_integral}, $T$ is a complete local domain and $S\to T$ is finite. As $T$ is a $1$-dimensional complete normal ring of equal characteristic, the Cohen structure theorem gives $T\cong k'\llbracket t\rrbracket$ for some field extension $k'/k$. Furthermore, the restriction of $S\to T$ to $k\to k'$ is finite and $k$ is algebraically closed, so $k = k'$. 

We define $\ell := \ord_t(xT)$, and note that by \Cref{lemma:root-vanishing} that $t^\ell\mid \theta_i$ for all $1\leq i\leq r$. Set $U:= T\llbracket y\rrbracket = k\llbracket y, t\rrbracket$. As $R\hookrightarrow U$ is a finite local map between regular local rings of the same dimension, $R\to U$ is flat by the miracle flatness lemma \cite[\href{https://stacks.math.columbia.edu/tag/00R4}{Tag 00R4}]{stacks-project}. Set $\bfr = \mf U = (y, t^\ell)U$. As $R\to U$ is flat, we have $\fpt(f) = \ft^{\mf}(f) = \ft^{\bfr}(f) = \frac{1}{d}$ by \Cref{lemma:flat-ft}. Write $d = qs$, where $s$ is coprime to $p$. We aim to show that there exists $g$ in $ U$ such that $g^s = f$. If $r = 1$ this is clear; we simply have $f = h_1^{qs}$ and $g = h_1^q$. Otherwise, suppose $r\geq 2$.

Let $\nf = (y,t)\subseteq U$. For all $1\leq i\leq r,$ we have $y-\theta_i\in \nf\setminus \nf^2$, so $h_i:= y - \theta_i$ is irreducible. As argued in \cite[Corollary 5.2]{sato_accumulation_2023}, the pair $(U, (h_1\dots h_r)^1)$ is sharply $F$-pure away from $\nf$, so by \cite[Theorem 5.1]{sato_accumulation_2023} there exists $D$ in $ \Z^+$ such that for all $t_1,\dots, t_r$ in $ [0, 1)$ and $\tilde{h},\dots, \tilde{h}_r$ such that $h_i\equiv \tilde{h}_i\mod \nf^D$, we have
\begin{equation}\label{eqn:sato-truncation}
\tau(U, h_1^{t_1}\dots h_r^{t_r}) = \tau(U, \tilde{h}^{t_1}\dots\tilde{h}_r^{t_r}).
\end{equation}
 Without loss of generality, choose $D$ to be larger than $\ord_{\nf}(\theta_i-\theta_j)$ for all $1\leq i<j\leq r$; in particular this value of $D$ satisfies $D\geq \ell$. For $1\leq i\leq r$, let $\tilde{\theta_i}$ in $ k[t]$ denote the truncation of the power series $\theta_i$ at the $D$th term; we have $\theta_i \equiv \tilde{\theta_i}\mod \nf^D$ and $t^\ell\mid \tilde{\theta_i}$. Set $\tilde{h_i} = y - \tilde{\theta_i}$ and $\tilde{f} = \tilde{h_1}^{e_1}\dots \tilde{h_r}^{e_r}$. By assumption that $r\geq 2$ we have $e_i \leq d-1$ for all $1\leq i\leq r$, hence $ce_i\in [0,1)$ for all $c$ in $[0, \frac{1}{d}]$. By \Cref{eqn:sato-truncation}, for all $c\in [0, \frac{1}{d}],$ we have
\[
\tau(R, f^c) = \tau(R, h_1^{ce_1}\dots h_r^{ce_r}) = \tau(R, \tilde{h_1}^{ce_1}\dots \tilde{h_r}^{ce_r}) = \tau(R, \tilde{f}^c). 
\]
As $\ft^{\bfr}(f) = \frac{1}{d}$, it follows that $\tau(R, \tilde{f}^c)\not\subseteq \bfr$ for all $0\leq c < \frac{1}{d}$ and $\tau(R, \tilde{f}^{1/d}) \subseteq \bfr$, so $\ft^{\bfr}(\tilde{f}) = \frac{1}{d}$.

By our choice of $D$, the factors $\tilde{h_1},\dots, \tilde{h_r}$ are distinct. In the notation of \Cref{defn:regions}, write $\tilde{f} = \tilde{\hb}^{\eb} = \tilde{h_1}^{e_1}\dots \tilde{h_r}^{e_r}$ with $e_i>0$. As $\ft^{\bfr}(\tilde{f})  = \frac{1}{\norm{\eb}}< \frac{1 + \ell/D}{\norm{\eb}}$, by \Cref{cor:critical-point-computes-ft}, there exists a critical point $\cb = \frac{\ab}{q_0} \leq \frac{1 + \ell/D}{\norm{\eb}}\eb$ such that $\ft^{\bfr}(\tilde{f}) = \max_{1\leq i\leq r}\frac{c_i}{e_i}$. We then have
     \begin{equation}\label{eqn:holder}
           1 = \norm{\eb}\cdot\ft^{\bfr}(\tilde{f}) =
           \sum_{i=1}^re_i\max\left(\frac{c_1}{e_1},\dots, \frac{c_r}{e_r}\right)\geq \sum_{i=1}^r c_i = \norm{\cb}. 
    \end{equation}
    By \Cref{lemma:norm-of-upper-region}, we have $\norm{\cb} \geq 1$, so the inequality in \Cref{eqn:holder} is an equality. In particular, $c_i = e_i\max(\frac{c_1}{e_1},\dots, \frac{c_r}{e_r})$ for all $1\leq i\leq r$, so $\frac{c_i}{e_i} = \max(\frac{c_1}{e_1},\dots, \frac{c_r}{e_r})$ for all $1\leq i\leq r$, hence $\cb = \ft^{\bfr}(\tilde{f}) \eb$. For all $1\leq i\leq r,$ we conclude that $\frac{e_i}{d} = \frac{a_i}{q_0}$. Recall that $d = qs$, where $s$ is coprime to $p$. Then $e_iq_0 = a_isq$, so we conclude that $s\mid e_i$. Consequently, $f$ has an $s$th root $g$ in $U$: similar to the $r=1$ case, we take $g = h_1^{e_1/s}\dots h_r^{e_r/s}$. 
    
    Recall that $x = ut^\ell$ for some unit $u$ in $ U^\times$. As $U$ is a power series ring over the algebraically closed field $k$,  we have $u^{1/\ell}\in U$. Consequently, $x^{1/\ell} = tu^{1/\ell}$ is conjugate to $t$, hence $U = k\llbracket y, x^{1/\ell}\rrbracket$. Set $L = \text{Frac}(k\llbracket y\rrbracket)$ and consider $f$ as an element of $ L\llbracket x\rrbracket$. As $g\in U\subseteq L\llbracket x^{1/\ell}\rrbracket$, it follows from \Cref{lemma:s-root} that $g\in L\llbracket x\rrbracket\subseteq \text{Frac}(R)$. As $R$ is integrally closed and $g$ is integral over $R$, it follows that $g\in R$.

    As $f = g^s$, by \Cref{prop:properties-of-ft} (2) we have $\ft^{\mf}(f) = \frac{1}{s}\ft^{\mf}(g)$. Additionally, we have $\ord_\mf(g) = \frac{1}{s}\ord_{\mf}(f)$, so $\ft^{\mf}(g) = \frac{1}{q} = \frac{1}{\ord_{\mf}(g)}$. By \Cref{lemma:degree-q-case}, we conclude that $g\in \mf^{[q]}$, proving the claim.
    \end{proof}

\begin{remark}
Using Schwede's results on centers of $F$-purity, one can give a short proof of \Cref{lemma:dim-2-alg-closed} in the special case that $d$ is coprime to $p$. Assume that $\fpt(f) = \frac{1}{d}$ and that $d$ is coprime to $p$; then $(p^e-1)\frac{1}{d}\in \Z^+$ for some $e > 0$. By \cite[Theorem 4.9]{hernandez_f-purity_2012}, the pair $(R, f^{1/d})$ is sharply $F$-pure but not strongly $F$-regular. Let $\pf$ be a center of sharp $F$-purity for $(R, f^{1/d})$ as in \cite{schwede_centers_2010}, which is a nonzero prime ideal by \cite[Proposition 4.6]{schwede_centers_2010}. We conclude $f^{\frac{p^e-1}{d}} \in (\pf^{[p^e]} : \pf)$ by \cite[Propositions 3.11 and 4.7]{schwede_centers_2010}. 

If $\codim(\pf) = 1$, then we have $\pf = (g)$ for some irreducible element $g$ in $R$. In this case, we have $f^{\frac{p^e-1}{d}} \in g^{p^e-1}$. As $p^e-1=\ord_\mf(f^{\frac{p^e-1}{d}}) \geq (p^e - 1)\ord_\mf(g)$, we must have $\ord_\mf(g) = 1$, so $q = 1$ and $g\in \mf^{[q]}$. By unique factorization of $f$ we have that $f = ug^d$ for some unit $u$ in $R$. As $k$ is algebraically closed, $u$ admits a $d$th root in $R$, proving the claim in this case. To finish the proof, we'll show that $\codim(\pf) \neq 2$. Suppose for the sake of contradiction that $\pf = \mf$. Then $f^{\frac{p^e-1}{d}}\in (x_1^{p^e}, x_2^{p^e}, (x_1x_2)^{p^e-1})\subseteq \mf^{p^e}$. But this contradicts the fact that $\ord_\mf(f^{\frac{p^e-1}{d}}) = p^e-1$, proving the claim.
\end{remark}
\subsection{Step (2)}
For this step, we require a local Bertini theorem due to Flenner \cite{flenner_bertini_1977}. 
\begin{lemma}[\cite{flenner_bertini_1977}, Satz 2.1]\label{lemma:flenner-Bertini}
     Let $(A, \mf)$ be a complete local ring with coefficient field $k$. Let $\Jf = (z_1,\dots, z_l)\subseteq A$ is a proper ideal. Suppose that $Q = \{\qf_1,\dots, \qf_r\}$ is a finite set of prime ideals of $D(\Jf)$. Then there exists $x$ in $\Jf$ such that:
     \begin{enumerate}[(i)]
         \item  $x\equiv z_1\mod (z_2,\dots, z_l) + \mf\Jf$;
         \item $ x\notin \qf_1\cup \dots \cup\qf_r$;
         \item $x\notin \pf^{(2)}$ for all $x\in D(\Jf)$.
     \end{enumerate}
\end{lemma}
\begin{proof}
While Flenner does not explicitly note the congruence condition on $x$, we can specify one detail of the construction to ensure (i) holds. In Flenner's notation, we assume $y_1,\dots, y_m$ is a generating set for $\mf$. We then let $x_1,\dots, x_n$ denote the set of elements $z_1,\dots, z_l, z_1y_1,\dots, z_ly_m$ with $x_i = z_i$ for $1\leq i\leq l$. With $S\subseteq A$ as in loc. cit., there exist $a_2,\dots, a_n$ in $S$ such that $x := x_1 + a_2x_2+\dots + a_nx_n$ satisfies (ii) and (iii). Noting that $x\equiv z_1 + a_2z_2 + \dots + a_lz_l \mod \mf \Jf$, the claim (i) follows.
\end{proof}

The following lemma is essentially an application of \cite[Korollar 3.5]{flenner_bertini_1977}, where we replace Satz 2.1 in op. cit. with our modification \Cref{lemma:flenner-Bertini}. Although Flenner's argument works for us \textit{mutatis mutandis}, because the work op. cit. is written in German, we spell out the changes explicitly for the reader's convenience.
\begin{lemma}\label{lemma:reduced-bertini-flenner}
    Let $k$ be an infinite field. Let $n\geq 3, R = k\llbracket x_1,\dots, x_n\rrbracket, \mf = (x_1,\dots, x_n)$. Let $f$ be an element of $R$ with $\ord_\mf(f) = d > 0$. Let $s > 0$ and suppose that $f$ does not have an $s$th root in $R$. There exists $x$ in $\mf\setminus \mf^2$ such that the image $\bar{f}$ of $f$ in $R/xR$ does not have an $s$th root in $R/xR$ and $\ord_{\mf/(x)}(\bar{f}) = d$.
\end{lemma}
\begin{proof}
Recall that a regular local ring is a unique factorization domain (UFD) \cite{auslander_UFD_1959}. Factor $f$ as $f = \pi_1^{e_1}\dots \pi_l^{e_l}$, where $\pi_1,\dots, \pi_l$ are distinct irreducible factors of $f$. By assumption that $f$ does not have an $s$th root in $R$, the multiplicities $e_i$ are not all divisible by $s$. Set $g = \pi_1\dots \pi_l$. Write $f = f_d + f_{>d}$, where $f_d$ is a homogeneous polynomial of degree $d$ and $f_{>d}\in \mf^{d+1}$. Let $Q\subseteq \Spec R$ denote the set of primes $\qf$ such that either $\qf$ is minimal over $(f_d, g)$ or such that $\codim(\qf) < n$ and the image of $\qf$ in $R/gR$ is a minimal element of $\Sing(R/gR)$. As $\dim R = n\geq 3$ and each of the minimal primes $\qf$ over $(f_d, g)$ has $\codim(\qf)\leq 2$, we conclude that $\mf\notin Q$.

By assumption that $k$ is infinite\footnote{It suffices to have $\#\P^{n-1}(k) > d$.}, there exists a homogeneous linear form in $R$ which is not a factor of $f_d$. Choose homogeneous coordinates $z_1,\dots, z_n$ for $R$ such that $z_1\nmid f_d$. Let $\Jf= (z_1, z_2^{d+1},\dots, z_n^{d+1})$. As $\sqrt{\Jf} = \mf$, we have that $Q\subseteq D(\Jf)$. Apply \Cref{lemma:flenner-Bertini} to $R, \Jf$ and $Q$ to produce $x$ in $\Jf$ such that $x\equiv z_1\mod (z_2^{d+1},\dots,z_n^{d+1}) + \mf\Jf, x\notin \bigcup_{\qf\in Q}\qf,
$ and $x\notin \pf^{(2)}$ for all $\pf \neq \mf$. Then $x\in \mf\setminus \mf^2$ by construction.

We study the image $\bar{f}$ of $f$ in $R/xR$, which we show has $\ord_{\mf/(x)}(\bar{f}) = d$. To see this, it suffices to show that $f\notin xR + \mf^{d+1}$. As $x\equiv z_1\mod (z_2^{d+1},\dots, z_n^{d+1}) + z_1\mf$, there exists $y$ in $\mf$ such that $x \equiv z_1 + yz_1\mod \mf^{d+1}$. As $z_1, f_d$ are homogeneous polynomials and $z_1\nmid f_d$, we have $f_d\notin z_1R + \mf^{d+1} = xR + \mf^{d+1}$, so $f\equiv f_d \not\equiv 0\mod xR+ \mf^{d+1}$ and thus $\ord_{\mf/(x)}(f) = d$.

Since $R/fR$ is reduced, the argument of \cite[Korollar 3.5]{flenner_bertini_1977} allows us to conclude that the ring $R/(f, x)R$ is reduced. In particular, the image of $f$ is squarefree in $R/xR$, so we may factor the image of each $\pi_i$ in $R/xR$ as $\bar{\pi_i} = \rho_{i1}\dots \rho_{is_i}$ where the $\rho_{ij}$ are irreducible and pairwise distinct. It follows that we may factor $\bar{f}$ as 
\(
\prod_{i=1}^l\prod_{j=1}^{s_i} \rho_{ij}^{e_i}.
\)
By assumption, the multiplicities $e_i$ are not all divisible by $s$. As $x\in \mf\setminus \mf^2$, the ring $R/xR$ is regular, hence a UFD, so $\bar{f}$ does not have an $s$th root in $R/xR$.
\end{proof}

\begin{lemma}\label{lemma:dim-n-alg-closed}
    Let $k$ be an algebraically closed field of characteristic $p>0$. Let $n\geq 1$ and let $R = k\llbracket x_1,\dots, x_n\rrbracket$. Let $f$ be an element of $R$ with $\ord_\mf(f)=d$. Write $d = qs$, where $q = p^e$ and $s$ is coprime to $p$. If $\fpt(f) = \frac{1}{d}$, then there exists $g$ in $\mf^{[q]}$ such that $f = g^s$.
\end{lemma}
\begin{proof}
    If $n = 1$, then $f = ux_1^d$ for some $u$ in $R^\times$. As $k = \bar{k}$, every unit in $R$ has a $d$th root in $R$, so we have $f = (u^{1/d}x_1)^d$ and $g = u^{q/d}x_1^q$. The $n=2$ case is \Cref{lemma:dim-2-alg-closed}; we now suppose the claim holds in dimension $n-1$ and let $n\geq 3$. Let $f$ be an element of $R$ with $\ord_\mf(f) = d$, and write $d = qs$ where $q = p^e$ and $\gcd(s,p)=1$. If $f$ does not have an $s$th root in $R$, then by \Cref{lemma:reduced-bertini-flenner} there exists $x$ in $\mf\setminus \mf^2$ such that, writing $\bar{f}$ for the image of $f$ in $R/xR$, the element $\bar{f}$ does not have an $s$th root in $R/xR$ and $\ord_{\mf/(x)}(\bar{f}) = d$. By the classification in dimension $n-1$, we have $\ft^{\mf/(x)}(\bar{f}) > \frac{1}{d}$. By \Cref{prop:properties-of-ft} (5), we have $\ft^{\mf}(f) \geq \ft^{\mf/(x)}(\bar{f}) > \frac{1}{d}$. 
    
    Suppose now that $\fpt(f) = \frac{1}{d}$. By the previous paragraph, $f$ admits an $s$th root $g$ in $R$. Writing $f = g^s$, by \Cref{prop:properties-of-ft} (2) and the fact that $\ord_{\mf}$ is a valuation, we have 
    \[\ft^{\mf}(g) = s\ft^{\mf}(f) = \frac{1}{q} = \frac{1}{\ord_{\mf}(g)}.\]
    By \Cref{lemma:degree-q-case}, we conclude that $g\in \mf^{[q]}$, proving the claim.
\end{proof}
\subsection{Step (3)} Over an arbitrary field $k$, an element $f$ in $k\llbracket x_1,\dots, x_n\rrbracket$ with $\ft^{\mf}(f) = \frac{1}{\ord_{\mf}(f)}$ may not have an $s$th root at all. 
\begin{example}\label{example:non-power}
    Let $R = k\llbracket x\rrbracket$. Let $q = p^e$ and $s > 0$ such that $\gcd(s, p) = 1$ and write $d = qs$. If $a\in k^\times$ such that $a$ does not have a $s$th root in $k$, then $\ft^{\mf}(ax^d) = \ft^{\mf}(x^d) = \frac{1}{d},$ but $ax^d$ does not have an $s$th root in $R$. 
\end{example}
The following lemma shows that \Cref{example:non-power} is as pathological as can possibly occur in a power series ring.
\begin{lemma}\label{lemma:unit-multiple-of-power}
    Let $k$ be a field, $R = k\llbracket x_1,\dots, x_n\rrbracket$, and set $S = R \widehat{\otimes}_k \bar{k}$. Let $f$ be an element of $R$ and $s$ in $\Z^+$ such that $\text{char}(k)$ does not divide $s$. If there exists $g$ in $S$ such that $g^s = f$, then there exists a unit $u$ in $k^\times$ and an element $h$ in $R$ such that $f = uh^s$.
\end{lemma}
\begin{proof}
    Let $\Gamma$ denote the set of monomials in $S$. Let $\succ$ be a local monomial order on $S$ such that $(\Gamma, \succ)$ is order-isomorphic to $(\N, >)$; for example, we may take $\succ$ such that $x^I\succ x^J\iff x^I <_{\text{deglex}} x^J$. Write $g = \sum_{i=0}^N a_ix^{I_i}$ where the $a_i$ are nonzero and $I_1 < I_2 < \dots$. We allow $N$ to be a positive integer or infinity; the proof is identical in both cases. As $a_1^s(x_1^{I_1})^s = \ini_\succ(g^s) = \ini_\succ(f) \in R$, it follows that $a_1^s\in k$. Set $u = a_1^{-s}$ and $h = a_1^{-1}g$; we have $f = uh^s$ and $h^s\in R$. To prove the claim, we'll show that $h\in R$.

    For ease of notation, write $b_i = \frac{a_i}{a_1}, i\geq 2$ so that $h = x^{I_1} + b_2x^{I_2} + \dots$. Suppose $j\geq 2$ such that $b_2,\dots, b_{j-1}\in k$ and consider the coefficient $c_j$ of $x^{(s-1)I_1 + I_j}$ in $h^s$. A priori, we have
    \begin{equation}\label{eqn:coefficient}
    c_j =\sum_{\substack{\ell_1,\dots, \ell_s\\ I_{\ell_1}+\dots + I_{\ell_s} = (s-1)I_1 + I_j}} b_{\ell_1}\dots b_{\ell_s}.
    \end{equation}
    If $\ell_1\leq \dots \leq \ell_s\in \N$ such that $I_{\ell_1} + \dots + I_{\ell_s} = (s-1)I_1 + I_j$, then we have $(s-1)I_1 + I_{\ell_s}\leq (s-1)I_1 + I_j$, so $\ell_s \leq j$ with equality if and only if $\ell_1 = \dots = \ell_{s-1} = 1$. Consequently, we may refine \Cref{eqn:coefficient}: 
    \begin{equation}\label{eqn:coefficient-refined}
    c_j= sb_j+\sum_{\substack{\ell_1,\dots, \ell_s < j\\ I_{\ell_1}+\dots + I_{\ell_s} = (s-1)I_1 + I_j}} b_{\ell_1}\dots b_{\ell_s}.
    \end{equation}
    As $h^s\in R$, the coefficient $c_j$ is in $k$. By assumption that $b_1,\dots, b_{j-1}\in k$, the second term of \Cref{eqn:coefficient-refined} is an element of $k$, so $sb_j\in k$. As $s$ is a unit in $k$, we conclude that $b_j\in k$.
\end{proof}
\begin{lemma}\label{lemma:dim-n-arbitrary-field}
    Let $k$ be a field of characteristic $p > 0, R = k\llbracket x_1,\dots, x_n\rrbracket,$ and $\mf = (x_1,\dots, x_n)$. Let $f$ be an element of $R$ such that $\ft^{\mf}(f) = \frac{1}{d}$. If $q$ is the largest power of $p$ dividing $d$ and $d = qs$, then there exists a unit $u$ in $R^\times$ and an element $g$ in $ \mf^{[q]}$ such that $f = ug^s$.
\end{lemma}
\begin{proof}
    Let $S = R\widehat{\otimes}_k \bar{k}$ and consider the map $R\to S$. We have that $R\to S$ is faithfully flat, $\nf = \mf S$ is the maximal ideal of $S$, and $S = \bar{k}\llbracket x_1,\dots, x_n\rrbracket$. By \Cref{lemma:flat-ft}, we have $\ft^{\mf}(f) = \ft^{\nf}(f)$ and $\ord_{\mf}(f) = \ord_{\nf}(f)$. As $\ft^{\nf}(f) = \frac{1}{d} = \frac{1}{\ord_{\nf}(f)}$, by \Cref{lemma:dim-n-alg-closed} there exists $g$ in $\mf^{[q]}$ such that $f = g^s$. By \Cref{lemma:unit-multiple-of-power}, there exist $h$ in $R, u$ in $ R^\times$ such that $uh^s = f$, and $h\in \mf^{[q]}$ by \Cref{lemma:degree-q-case}. 
\end{proof}
\subsection{Step (4)} 
\begin{lemma}\label{lemma:arb-ideal-arb-field}
    Let $k$ be a field of characteristic $p > 0, R = k\llbracket x_1,\dots, x_n\rrbracket,$ and $\mf = (x_1,\dots, x_n)$. Let $\af\subseteq R$ such that $\ft^{\mf}(\af) = \frac{1}{d}$. If $q$ is the largest power of $p$ dividing $d$ and $d = qs$, then there exists $g$ in $\mf^{[q]}$ such that $g^sR = \af$.
\end{lemma}
\begin{proof}
    Let $f$ be an element of $\af$ such that $\ord_{\mf}(f) = d$. By \Cref{prop:properties-of-ft} (1), (3) we have $\frac{1}{d}\leq \ft^{\mf}(f)\leq \ft^{\mf}(\af) = \frac{1}{d}$. By \Cref{lemma:dim-n-arbitrary-field}, there exists $u$ in $R^\times, g$ in $ \mf^{[q]}$ such that $f = ug^s$. We will show that $g^s$ generates $\af$. 

    Let $e > 0$ such that $s\mid (p^e - 1)$; write $t_e = \frac{p^e-1}{s}$. By \Cref{prop:properties-of-ft} (4), for all $e > 0$ we have $\nu_{\af}^{\mf}(qp^e) < qp^e\ft^{\mf}(\af) \leq \ceil{qp^e/qs} = t_e+1$, so $\af^{t_e+1}\subseteq \mf^{[p^e]}$ for all $e > 0$. Let $z$ be an arbitrary element of $\af$ and let $r$ be a positive integer. As $\af^{t_e+1}\subseteq \mf^{[qp^e]}$, we in particular have $(g^s)^{t_e-r}z^{r+1}\in \mf^{[qp^e]}$. Write $g = a_1x_1^q + \dots + a_nx_n^q$. As=  $\ord_{\mf}(g) = q$, there is some $1\leq i\leq n$ such that $a_i\in R^\times$. As $g\equiv a_ix_i^q\mod (x_1^q,\dots, x_{i-1}^q,x_{i+1}^q,\dots, x_n^q)$, it follows that $\mf^{[q]} = (x_1^q,\dots, x_{i-1}^q, g, x_{i+1}^q, \dots, x_n^q)$. In particular, $x_1,\dots, x_{i-1},g,x_{i+1},\dots, x_n$ is a system of parameters for $R$. By \cite[Lemma 3.5]{kadyrsizova_lower_2022}, we have 
    \begin{equation}\label{eqn:ord_g_z_r}
    \begin{aligned}
        z^{r+1}\in (\mf^{[qp^e]}:(g^s)^{t_e-r}) &= \left((x_1^{qp^e},\dots,g^{p^e},\dots, x_n^{qp^e}):g^{p^e-1-sr}\right) \\&=  (x_1^{qp^e},\dots, g^{1+sr},\dots, x_n^{qp^e}).
    \end{aligned}
    \end{equation}
    Applying \Cref{eqn:ord_g_z_r} and letting $e\to\infty$, we obtain $z^{r+1}\in g^{1+sr}R\subseteq g^{sr}R$ for all $r > 0$. By \cite[Corollary 6.8.11]{huneke_integral}, $z$ is contained in the integral closure of the ideal $g^sR$, which by \cite[Proposition 1.5.2]{huneke_integral} is equal to $g^sR$ itself.
\end{proof}
\subsection{Step (5)}
\begin{lemma}\label{lemma:principal-descends}
    Let $(A, \mf)$ be a regular local ring and $I\subseteq A$ an ideal such that $I\widehat{A}$ is principal. Then $I$ is principal. 
\end{lemma}
\begin{proof}
    Let $M$ be an $A$-module. Since $I\widehat{A}$ is principal, $I\widehat{A}$ is flat, so $\Tor^{\widehat{A}}_1(I\widehat{A}, M\otimes_A \widehat{A}) = 0$. Consequently, by \cite[\href{https://stacks.math.columbia.edu/tag/00M8}{Tag 00M8}]{stacks-project} we have
    \[
    0 = \Tor^{\widehat{A}}_1(I\widehat{A}, M\otimes_A \widehat{A}) = \Tor_1^{\widehat{A}}(I\otimes_A \widehat{A}, M\otimes_A \widehat{A}) = \Tor_1^A(I, M)\otimes_A \widehat{A},
    \]
    so by faithful flatness of $A\to \widehat{A}$ we conclude $\Tor^A_1(I, M) = 0$. As $M$ was arbitrary, we deduce that $I$ is a flat $A$-module, hence $I$ is a principal ideal.
\end{proof}
\begin{lemma}\label{lemma:good-formal-fibers}
Let $(A, \mf)$ be a regular local ring and $I\subseteq A$ an ideal such that $I\widehat{A}$ is principal and generated by an element $g^s$ for some $s$ in $\Z^+$, $g$ in $\widehat{A}$. Suppose that for all prime elements $\pi\in A$, the formal fiber $\widehat{A}\otimes_A A_{(\pi)}/\pi A_{(\pi)}$ is reduced. Then there exists $h$ in $A$ such that $I = h^sA$.
\end{lemma}
\begin{proof}
    By \Cref{lemma:principal-descends}, choose $f$ in $A$ such that $I = fA$. As $I\widehat{A} = f\widehat{A} = g^s\widehat{A}$, it follows from \Cref{lemma:unit-multiple-of-power} that $f = ug^s$ for some $u$ in $\widehat{A}^\times$.

    Recall again that $A, \widehat{A}$ are UFDs \cite{auslander_UFD_1959}. Factor $f$ in $A$ as $f = v\pi_1^{d_1}\dots \pi_l^{d_l}$, with $v$ in $A^\times$ and $\pi_1,\dots, \pi_l$ pairwise coprime and irreducible. In $\widehat{A},$ factor each $\pi_i$ as $\rho_{i1}^{e_{i1}}\dots \rho_{is_i}^{e_{is_i}}$ where each $\rho_{ij}$ is irreducible and $\rho_{ij},\rho_{ij'}$ are coprime for $j\neq j'$. By assumption that the formal fiber $\widehat{A}\otimes_A A_{(\pi_i)}/\pi_iA_{(\pi_i)}$ is reduced for all $1\leq i\leq l$, it follows that $e_{ij} = 1$ for all $1\leq i\leq l, 1\leq j\leq s_i$. Moreover, as $\rho_{ij}\widehat{A}\cap A =\pi_i A$, it follows that the $\rho_{ij}, \rho_{i'j'}$ are coprime for all $(i,j)\neq (i', j')$. Writing 
    \[
    ug^s = f = v\pi_1^{d_1}\dots \pi_l^{d_l} = \prod_{i=1}^l
\prod_{j=1}^{s_i}\rho_{ij}^{d_i},\]
we conclude that $s\mid d_i$ for all $1\leq i\leq l$. Setting $h = \pi_1^{d_1/s}\dots \pi_l^{d_l/s}$, we have $I = h^sA$.
\end{proof}
\begin{proof}[Proof of \Cref{thm:E1-char-p}]
    Let $(A, \mf)$ be a regular local ring of characteristic $p > 0$ and $\af\subseteq A$ an ideal with $\ord_{\mf}(\af) = d>0$. Factor $d$ as $qs$, where $q = p^e$ and $\gcd(p,s)=1$. Suppose $\ft^{\mf}(\af) = \frac{1}{d}$. By \Cref{lemma:flat-ft} we have 
    \[ \ft^{\mf \widehat{A}}(\af \widehat{A})=\ft^{\mf}(\af) = \frac{1}{d} = \frac{1}{\ord_{\mf}(\af)} = \frac{1}{\ord_{\nf}(\af \widehat{A})}.\]
    By \Cref{lemma:arb-ideal-arb-field}, there exists $g$ in $ \mf^{[q]}\widehat{A}$ such that $g^s = \af \widehat{A}$. If we additionally assume that for all prime elements $\pi$ of $A$ that the formal fiber $\widehat{A}\otimes_A A_{(\pi)}/\pi A_{(\pi)}$ is reduced, then by \Cref{lemma:good-formal-fibers} there exists $h$ in $\mf^{[q]}$ such that $\af = h^sA$.

    Conversely, suppose that $\af \widehat{A} = g^s\widehat{A}$ for some $g$ in $\mf^{[q]}\widehat{A}$. By \Cref{lemma:flat-ft,lemma:degree-q-case} we have $\ft^{\mf}(\af) = \frac{1}{d}$. Similarly, if there exists $h$ in $\mf^{[q]}$ such that $\af = h^sA$, then $\ft^{\mf}(\af) = \frac{1}{d}$ by \Cref{lemma:degree-q-case}.
\end{proof}

Our hypothesis on the formal fibers is necessary by \Cref{example:formal-fiber-ctrexmpl}, which relies on the following lemma.
\begin{lemma}\label{lemma:bad-formal-fibers}
    Let $(T, \mf)$ be a complete local domain of dimension at least $2$ which satisfies Serre's condition $S_2$. Suppose $T$ has a coefficient field $k$. Let $x$ be a nonzero element of $\mf$. Then there exists a Noetherian local subring $(A, \mf\cap A)$ such that $x\in A, \widehat{A} = T,$ and $x$ is a prime element of $A$.
\end{lemma}
\begin{proof}
    Let $C = \{\qf_1,\dots, \qf_r\}$ denote the set of minimal primes over $x$. Let $\Pi$ denote the prime subring of $T$ -- that is, $\Pi = \Q$ if $\char k = 0$ and $\Pi = \F_p$ if $\char k = p>0$. The claim follows once we check that this setup satisfies the hypotheses of \cite[Theorem 1.1]{chatlos_fibers_2012}.
    \begin{enumerate}
    \setcounter{enumi}{-1}
        \item The primes $\qf_1,\dots, \qf_r$ have height at most 1, and are thus nonmaximal. Moreover, we have $x\in \qf_1\cap \dots \cap \qf_r$.
        \item As $\Ass(T) = \{(0)\}$, we trivially have $\pf\cap \Pi[x] = (0)$ for all $\pf$ in $ \Ass(T)$.
        \item As $T$ is $S_2$ and $x$ is a regular element, it follows that $T/xT$ is $S_1$. Consequently, the associated primes of $T/xT$ are precisely the minimal primes of $T/xT$, so $\Ass(T/xT) = C$. 
        \item For all $1\leq i\leq r$, we claim that $\text{Frac}(\Pi[x])\cap \qf_i \subseteq xT$. To see this, suppose that there exist $d,e$ in $ \Z_{\geq 0}, a_1,\dots, a_d, b_1,\dots, b_e$ in $ \Pi, f$ in $\qf_i$ such that
        \begin{equation}\label{eqn:pi_x_intersect_qf_i}
        \frac{x^d + a_1x^{d-1}+\dots + a_d}{x^e+b_1x^{e-1}+\dots+b_e} = f. 
        \end{equation}
        Moreover, we suppose that the left-hand side of \Cref{eqn:pi_x_intersect_qf_i} is written in reduced form. Then we have
        \[
        a_d = x(x^{d-1} + a_1x^{d-2} +\dots+ a_{d-1}) -f(x^e+b_1x^{e-1}+\dots+b_e) \in \mf.
        \]
        As $a_d\in \Pi\cap \mf$, it follows that $a_d = 0$, so we have
        \[
        b_ef = x(x^{d-1} + a_1x^{d-2} +\dots+ a_{d-1} - f(x^{e-1}+b_1x^{e-2}+\dots+b_{e-1}))\in xT.
        \]
        As the left-hand side of \Cref{eqn:pi_x_intersect_qf_i} is written in reduced form and $x$ divides the numerator $x^d + \dots + xa_{d-1}$, it follows that $b_e\neq 0$, so $f\in xT$.
    \end{enumerate}
\end{proof}
As an application of the above lemma, we show that for all $d >0$, there exist many regular local rings $(A, \mf)$ and prime elements $f$ in $A$ such that $\ord_{\mf}(f) = d$, $\ft^{\mf}(f) = \frac{1}{d}$.
\begin{example}\label{example:formal-fiber-ctrexmpl}
    Let $k$ be a field of characteristic $p>0$. For any $n\geq 2$ let $T = k\llbracket x_1,\dots, x_n\rrbracket$ and $\mf = (x_1,\dots, x_n)$. Fix some $e\geq 0, q = p^e$, and $s>0$ such that $\gcd(s,p)=1$. Choose $g$ in $ \mf^{[q]}$ such that $\ord_{\mf}(g) = q$ and set $f = g^s$. By \Cref{lemma:degree-q-case}, we have $\ft^{\mf}(f) = \frac{1}{s}\ft^{\mf}(g) = \frac{1}{qs} = \frac{1}{\ord_{\mf}(f)}$.
    
    By \Cref{lemma:bad-formal-fibers}, there exists a Noetherian local subring $(A, \mf\cap A)\subseteq T$ such that $f\in A, \widehat{A} = T,$ and such that $f$ is a prime element of $A$. Let $\nf:= \mf\cap A$. As argued in \Cref{lemma:dim-n-arbitrary-field}, we have $\ft^{\nf}(f) = \ft^{\mf}(f)$ and $\ord_{\nf}(f) = \ord_{\mf}(f)$, so $\ft^{\nf}(f) = \frac{1}{\ord_{\nf}(f)}$.
\end{example}
To conclude this article, we consider the effect on the main theorem of adding an additional reducedness hypothesis. 
\begin{example}\label{example:E1-char-p-reduced}
    If $R$ is a regular local ring, $x,y$ part of a regular system of parameters for $R$, and $a,b$ coprime integers, then $\widehat{R}/(x^a - y^b)\widehat{R}$ is geometrically integral. To see this, writing $\widehat{R} = k\llbracket x, y, z_1,\dots, z_n\rrbracket$, we have $\widehat{R}/(x^a - y^b)\widehat{R}\cong k\llbracket t^a,t^b,z_1,\dots, z_n\rrbracket$.

    Let $(R, \mf)$ be a regular local ring of characteristic $p > 0$. Let $\af\subseteq R$ such that
    $\ord_{\mf}(\af) = d, \ft^{\mf}(\af) = \frac{1}{d},$ and $\widehat{R}/\af\widehat{R}$ is reduced. Factor $d$ as $d = qs$, where $q = p^e$ and $\gcd(s,p)=1$. By \Cref{thm:E1-char-p}, there exists $g$ in $\mf^{[q]}\widehat{R}$ such that $\af\widehat{R} = g^s\widehat{R}$, so reducedness of $\widehat{R}/g^s\widehat{R}$ forces $s = 1$. By \Cref{lemma:intersection-pure} we also have $\af\subseteq \mf^{[q]}$. For any $q = p^e$, if $x, y$ is part of a regular system of parameters for $R,$ then with $f =x^q - y^{q+1}$, we have $\ord_{\mf}(f) = q$ by construction, $\ft^{\mf}(f_q) = \frac{1}{q}$ by \Cref{lemma:degree-q-case}, and $\widehat{R}/f_q\widehat{R}$ is \textit{geometrically} reduced. 

    To show that $\frac{1}{d}$ is an optimal lower bound on $\ft^{\mf}(\af)$ when $\ord_{\mf}(\af) = d$ and $\widehat{R}/\af  \widehat{R}$ is reduced (or geometrically reduced), suppose $\dim R\geq 2$ and $x,y$ is part of a regular system of parameters for $R$. Setting $g_{d,t} = x^d - y^{td+1}$, we have $\ft^{\mf}(g_{d,t}) \leq \ft^{\mf}((x^d, y^{td+1})) \leq \frac{1}{d} + \frac{1}{td+1}\searrow \frac{1}{d}$, where the inequality is by \Cref{prop:properties-of-ft} (1) and the equality by \cite[Proposition 36]{hernandez_f-purity_2016}. 
\end{example}
\printbibliography
\end{document}